\documentclass[11pt]{article}
\usepackage{amssymb}
\usepackage{amsmath}
\usepackage{amsthm} 
\usepackage{color}
\usepackage{wrapfig}

\usepackage{float}
\usepackage{epic}   
\usepackage{times}

\allowdisplaybreaks[4] 

\makeatletter
\def\section{\@startsection{section}{1}{\z@}{-4.5ex plus -1ex minus
    -.2ex}{2.7ex plus .2ex}{\Large\bf}}

\newtheorem{lemma}{Lemma}
\newtheorem{corollary}{Corollary}
\newtheorem{remark}{Remark}
\newtheorem{theorem}{Theorem}
\numberwithin{equation}{section}

\renewenvironment{proof}{ 
\noindent{\bf Proof.} \rm}{\penalty-20\null\hfill $\square$}

\newcommand{\bbF}{{\mathbb F}}
\newcommand{\bbH}{{\mathbb H}}
\newcommand{\bbI}{{\mathbb I}}

\newcommand{\N}{{\mathbb N}}
\newcommand{\bbO}{{\mathbb O}}
\newcommand{\R}{{\mathbb R}}

\newcommand{\bfe}{\mathbf{e}}
\newcommand{\bff}{\mathbf{f}}
\newcommand{\bfg}{\mathbf{g}}
\newcommand{\bfh}{\mathbf{h}}

\newcommand{\bfn}{\mathbf{n}}

\newcommand{\bfr}{\mathbf{r}}

\newcommand{\bfu}{\mathbf{u}}
\newcommand{\bfv}{\mathbf{v}}
\newcommand{\bfw}{\mathbf{w}}
\newcommand{\bfx}{\mathbf{x}}

\newcommand{\bfF}{\mathbf{F}}
\newcommand{\bfG}{\mathbf{G}}

\newcommand{\bfL}{\mathbf{L}}

\newcommand{\bfV}{\mathbf{V}}
\newcommand{\bfW}{\mathbf{W}}

\newcommand{\cA}{{\cal A}}

\newcommand{\cC}{{\cal C}}

\newcommand{\cO}{{\cal O}}

\newcommand{\rmd}{\mathrm{d}}
\newcommand{\rme}{\mathrm{e}}

\newcommand{\br}{\hspace{0.7pt}}
\renewcommand{\div}{\mathrm{div}\,}

\newcommand{\bfzero}{\mathbf{0}}
\newcommand{\dist}{\mathrm{dist}}

\newcommand{\supp}{\mathrm{supp}\,}

\newcommand{\Gammai}{\Gamma_{\hspace{-1.1pt} \rm in}}
\newcommand{\Gammao}{\Gamma_{\hspace{-1.1pt} \rm out}}
\newcommand{\Gammaw}{\Gamma_{\hspace{-1.1pt} P}}
\newcommand{\Gammam}{\Gamma_{\hspace{-1pt} 0}}
\newcommand{\Gammap}{\Gamma_{\hspace{-1pt} 1}}
\newcommand{\gammai}{\gamma_{\rm in}}
\newcommand{\gammao}{\gamma_{\rm out}}
\newcommand{\Omegah}{\widehat{\Omega}}
\newcommand{\Am}{A_0}
\newcommand{\Ap}{A_1}
\newcommand{\Bm}{B_0}
\newcommand{\Bp}{B_1}

\newcommand{\bfgs}{\bfg_{\displaystyle *}}

\newcommand{\cCs}{\boldsymbol{\cC}^{\infty}_{\sigma}(\Omega)}

\newcommand{\Vs}[1]{\bfV_{\sigma}^{1,{#1}}\hspace{-0.6pt}(\Omega)}
\newcommand{\Vsd}[1]{\bfV_{\sigma}^{-1,{#1}}\hspace{-0.6pt}(\Omega)}
\newcommand{\vs}[1]{\bfV_{\sigma}^{1,{#1}}}
\newcommand{\vsd}[1]{\bfV_{\sigma}^{-1,{#1}}}
\newcommand{\Vsdelta}[1]{\bfV_{\sigma}^{1,{#1}}\hspace{-0.6pt}
     (\Omega^{\delta})}
\newcommand{\Vsddelta}[1]{\bfV_{\sigma}^{-1,{#1}}\hspace{-0.6pt}
     (\Omega^{\delta})}
\newcommand{\bfWp}{\bfW_{\rm per}}

\newcommand{\bfvt}{\widetilde{\bfv}}
\newcommand{\vt}{\widetilde{v}}
\newcommand{\pt}{\widetilde{p}}
\newcommand{\bfft}{\widetilde{\bff}}
\newcommand{\bfht}{\widetilde{\bfh}}

\newcommand{\hht}{\widetilde{h}}
\newcommand{\bbFt}{\widetilde{\bbF}}

\newcommand{\bfFt}{\widetilde{\bfF}}

\newcommand{\lclosed}{[\hbox to 1pt{}}
\newcommand{\rclosed}{\hbox to 1pt{}]}

\newcommand{\blangle}{\bigl\langle}
\newcommand{\brangle}{\bigr\rangle}

\newcounter{constants}
\newcommand{\cn}[2]{ \addtocounter{constants}{1}
\newcounter{c#1#2}
\setcounter{c#1#2}{\value{constants}} c_{\arabic{c#1#2}} }
\newcommand{\cc}[2]{c_{\arabic{c#1#2}}}

\definecolor{lightgrey}{rgb}{0.82,0.82,0.82}

\begin{document}

\title{\LARGE \bf The maximum regularity property of the steady Stokes
problem associated with a flow through a profile cascade in
$L^r$--framework}

\author{Tom\'a\v s Neustupa}

\date{}

\maketitle

\begin{abstract}
The paper deals with the Stokes problem, associated with a flow of
a viscous incompressible fluid through a spatially periodic
profile cascade. We use results from \cite{TNe4} (the maximum
regularity property in the $L^2$--framework) and \cite{TNe5} (the
weak solvability in $W^{1,r}$), and extend the findings on the
maximum regularity property to the general $L^r$--framework (for
$1<r<\infty$). Using the reduction to one spatial period $\Omega$,
the problem is formulated by means of boundary conditions of three
types: the conditions of periodicity on curves $\Gammam$ and
$\Gammap$, the Dirichlet boundary conditions on $\Gammai$ and
$\Gammaw$ and an artificial ``do nothing''--type boundary
condition on $\Gammao$ (see Fig.~1). We show that, although domain
$\Omega$ is not smooth and different types of boundary conditions
``meet'' in the vertices of $\partial\Omega$, the considered
problem has a strong solution with the maximum regularity property
for ``smooth'' data. We explain the sense in which the ``do
nothing'' boundary condition is satisfied for both weak and strong
solutions.
\end{abstract}

\noindent
{\it AMS math.~classification (2000):} \ 35Q30, 76D03, 76D05.


\section{Introduction} \label{S1}

{\bf One spatial period: domain $\Omega$.} \ Mathematical models
of a flow through a three--dimensional turbine wheel often use the
reduction to two space dimensions, where the flow
\begin{wrapfigure}[13]{r}{81mm}
\begin{center}
  \setlength{\unitlength}{0.5mm}
  \begin{picture}(162,125)
  \put(5,60){\vector(1,0){151}} \put(20,37){\vector(0,1){95}}
  \put(143,64){$x_1$} \put(23,125){$x_2$}
  \put(28.5,106.4){\vector(0,1){3.2}}
  \dashline[+30]{2.2}(28.5,104.9)(28.5,75.2)
  \dashline[+30]{2.2}(130.1,70)(130.1,96)
  \dashline[+30]{2.2}(130.1,108)(130.1,132)
  \dashline[+30]{2.2}(130.1,40)(130.1,31)
  \put(118,100){$x_1=d$}
  \put(28.5,73.6){\vector(0,-1){3.7}} \put(31,97){$\tau$}
  \put(23,45){$\gammai$} \put(133,123){$\gammao$}
  \thicklines 
  \color{lightgrey}
\put(40.2,90){\line(1,0){35}} \put(40.2,90.4){\line(1,0){34}}
\put(40.2,90.8){\line(1,0){33}} \put(40.4,91.2){\line(1,0){31.6}}
\put(40.4,91.6){\line(1,0){30.4}} \put(40.5,92){\line(1,0){29.3}}
\put(40.5,92.4){\line(1,0){28.1}}\put(40.8,92.8){\line(1,0){26.5}}
\put(41.0,93.2){\line(1,0){25.0}}\put(41.5,93.6){\line(1,0){23.3}}
\put(41.8,94.0){\line(1,0){21.3}}\put(42.1,94.4){\line(1,0){19.3}}
\put(42.5,94.8){\line(1,0){17.3}}\put(43.2,95.2){\line(1,0){14.8}}
\put(44.3,95.6){\line(1,0){11.8}}\put(46.3,96.0){\line(1,0){6.8}}
\put(40.2,89.6){\line(1,0){36.1}}\put(40.4,89.2){\line(1,0){36.7}}
\put(40.6,88.8){\line(1,0){37.4}}\put(40.8,88.4){\line(1,0){38.2}}
\put(41.1,88.0){\line(1,0){38.9}}\put(41.5,87.6){\line(1,0){39.4}}
\put(42.1,87.2){\line(1,0){39.5}}\put(42.7,86.8){\line(1,0){39.8}}
\put(43.7,86.4){\line(1,0){39.7}}\put(44.5,86){\line(1,0){39.6}}
\put(45.7,85.6){\line(1,0){39.3}}\put(46.8,85.2){\line(1,0){39.1}}
\put(48.5,84.8){\line(1,0){38.2}}\put(50.4,84.4){\line(1,0){37.0}}
\put(52.3,84.0){\line(1,0){35.8}}\put(55.0,83.6){\line(1,0){33.8}}
\put(58.3,83.2){\line(1,0){31.2}}\put(61.7,82.8){\line(1,0){28.6}}
\put(65.3,82.4){\line(1,0){25.6}}\put(68.0,82.0){\line(1,0){23.5}}
\put(70.2,81.6){\line(1,0){21.9}}\put(72.2,81.2){\line(1,0){20.9}}
\put(74.2,80.8){\line(1,0){19.4}}\put(75.9,80.4){\line(1,0){18.3}}
\put(77.5,80.0){\line(1,0){17.3}}\put(79.0,79.6){\line(1,0){16.3}}

\put(80.8,79.2){\line(1,0){15.3}}\put(82.3,78.8){\line(1,0){14.4}}
\put(84.0,78.4){\line(1,0){13.3}}\put(85.5,78.0){\line(1,0){12.6}}
\put(87,77.6){\line(1,0){11.7}}\put(88.5,77.2){\line(1,0){10.8}}
\put(90.0,76.8){\line(1,0){9.9}}\put(91,76.4){\line(1,0){9.3}}
\put(92.2,76.0){\line(1,0){8.8}}\put(93.4,75.6){\line(1,0){8.2}}
\put(94.4,75.2){\line(1,0){7.4}}\put(95.8,74.8){\line(1,0){6.7}}
\put(97.1,74.4){\line(1,0){5.6}}\put(98.3,74.0){\line(1,0){4.6}}
\put(99.5,73.6){\line(1,0){3.5}}\put(100.3,73.2){\line(1,0){3.0}}
\color{black} 
\qbezier(40,90)(40,100)(63,94) \qbezier(40,90)(40,85)(60,83)
\qbezier(63,94)(86,87)(98,78) \qbezier(60,83)(80,80.5)(97.2,74.3)
\qbezier(98,78)(109,70)(97.2,74.3) \put(51,88){$P$}
  \thicklines 
  \put(20,69){\line(0,1){40}} \put(130.2,40){\line(0,1){40}}
  \qbezier(20,69)(70,76)(130,40) \qbezier(20,109)(70,116)(130,80)
  \thinlines
  \put(9.5,89){$\Gammai$} \put(116.2,67){$\Gammao$}
  \put(75,107){$\Gammap$} \put(75,67){$\Gammam$}
  \put(37,79){$\Gammaw$}
  \put(109,78){$\Omega$}
  \put(8,67){$\Am$} \put(8,107){$\Ap$}
  \put(133,39){$\Bm$} \put(133,78){$\Bp$}
  \thicklines
  \put(43,23){Fig.~1: \ Domain $\Omega$}
  \end{picture}
\end{center}
\end{wrapfigure}
is studied as a flow through an infinite planar profile cascade.
In an appropriately chosen Cartesian coordinate system, the
profiles in the cascade periodically repeat with the period $\tau$
in the $x_2$--direction. It can be naturally assumed that the flow
is $\tau$--periodic in variable $x_2$, too. This enables one to
study the flow through one spatial period, which contains just one
profile -- see domain $\Omega$ and profile $P$ on Fig.~1. This
approach is used e.g.~in \cite{DFF}, \cite{KLP}, \cite{SPKF},
where the authors present the numerical analysis of the models or
corresponding numerical simulations, and in the papers,
\cite{FeNe1}--\cite{FeNe3} and \cite{TNe1}--\cite{TNe3}, devoted
to theoretical analysis of the mathema\-tical models.

We assume that a viscous incompressible fluid flows into the
cascade through the straight line $\gammai$ (the $x_2$--axis, the
inflow) and essentially leaves the cascade through the straight
line $\gammao$, whose equation is $x_1=d$ (the outflow). By
``essentially'' we mean that we do not exclude possible reverse
flows on the line $\gammao$. The parts of $\partial\Omega$ (the
boundary of $\Omega$), lying on the straight lines $\gammai$ and
$\gammao$ are the line segments $\Gammai\equiv\Am\Ap$ and
$\Gammao\equiv\Bm\Bp$ of length $\tau$, respectively. The other
parts of $\partial\Omega$ are denoted by $\Gammaw$ (the boundary
of profile $P$), $\Gammam$ and $\Gammap\equiv\Gammam+\tau
\br\bfe_2$, see Fig.~1. (We denote by $\bfe_2$ is the unit vector
in the $x_2$--direction.) We may assume, without loss of
generality, that domain $\Omega$ is Lipschitzian and the curves
$\Gammam$ and $\Gammap$ are of the class $C^{\infty}$.

\vspace{4pt} \noindent
{\bf The Stokes boundary--value problem on one spatial period.} \
The fluid flow is described by the Navier-Stokes equations. An
important role in theoretical studies of these equations play the
properties of solutions to the steady Stokes problem. The steady
Stokes equation, which comes from the momentum equation in the
Navier--Stokes system if one neglects the derivative with respect
to time and the nonlinear ``convective'' term, has the form
\begin{equation}
-\nu\Delta\bfu+\nabla p\ =\ \bff. \label{1.1}
\end{equation}
It is studied together with the equation of continuity (=
condition of incompressibility)
\begin{equation}
\div\bfu\ =\ 0. \label{1.2}
\end{equation}
The unknowns are $\bfu=(u_1,u_2)$ (the velocity) and $p$ (the
pressure). The positive constant $\nu$ is the kinematic
coefficient of viscosity and $\bff$ denotes the external body
force. The density of the fluid can be without loss of generality
supposed to be equal to one. The system (\ref{1.1}), (\ref{1.2})
is completed by appropriate boundary conditions on
$\partial\Omega$. One can naturally assume that the velocity
profile on $\Gammai$ is known, which leads to the inhomogeneous
Dirichlet boundary condition
\begin{equation}
\bfu\ =\ \bfg \qquad \mbox{on}\ \Gammai. \label{1.3}
\end{equation}
Further, we consider the homogeneous Dirichlet boundary condition
\begin{equation}
\bfu\ =\ \bfzero \qquad \mbox{on}\ \Gammaw \label{1.4}
\end{equation}
and the conditions of periodicity on $\Gammam$ and $\Gammap$
\begin{align}
\bfu(x_1,x_2+\tau)\ &=\ \bfu(x_1,x_2) && \mbox{for}\
\bfx\equiv(x_1,x_2)\in\Gammam, \label{1.5} \\ \noalign{\vskip 4pt}
\frac{\partial\bfu}{\partial\bfn}(x_1,x_2+\tau)\ &=\
-\frac{\partial\bfu}{\partial\bfn}(x_1,x_2) &&
\mbox{for}\ \bfx\equiv(x_1,x_2)\in\Gammam, \label{1.6} \\
\noalign{\vskip 4pt}
p(x_1,x_2+\tau)\ &=\ p(x_1,x_2) && \mbox{for}\
\bfx\equiv(x_1,x_2)\in\Gammam. \label{1.7}
\end{align}
Finally, we consider the artificial boundary condition
\begin{equation}
-\nu\, \frac{\partial\bfu}{\partial\bfn}+p\br\bfn\ =\ \bfh \qquad
\mbox{on}\ \Gammao, \label{1.8}
\end{equation}
where $\bfh$ is a given vector--function on $\Gammao$ and $\bfn$
denotes the unit outer normal vector, which is equal to
$\bfe_1\equiv(1,0)$ on $\Gammao$. The boundary condition
(\ref{1.8}) (with $\bfh=\bfzero$) is often called the ``do
nothing'' condition, because it naturally follows from an
appropriate weak formulation of the boundary--value problem, see
\cite{Glow} and \cite{HeRaTu}.

\vspace{4pt} \noindent
{\bf On some previous related results.} \ In studies of the
Navier--Stokes equations in channels or profile cascades with
artificial boundary conditions on the outflow, many authors use
various modifications of condition (\ref{1.8}). (See
e.g.~\cite{BrFa}), \cite{FeNe1}, \cite{FeNe2}, \cite{FeNe3},
\cite{TNe1}, \cite{TNe2}, \cite{TNe3}.) The reason is that, while
condition (\ref{1.8}) does not enable one to control the amount of
kinetic energy in $\Omega$ in the case of a reverse flow on
$\Gammao$, the modifications are suggested so that one can derive
an energy inequality, and consequently prove the existence of weak
solutions. In papers \cite{KuSka} and \cite{Ku}, the authors use
the boundary condition on an outflow in connection with a flow in
a channel, and they prove the existence of weak solutions  of the
Navier--Stokes equations for ``small data''. Possible reverse
flows (again on an ``outflow'' of a channel) are controlled by
means of additional conditions in \cite{KraNe1}, \cite{KraNe2},
\cite{KraNe3}, where the Navier--Stokes equations are replaced by
the Navier--Stokes variational inequalities.

The regularity up to the boundary of existing weak solutions
(stationary or time--dependent) to the Navier--Stokes equations
with the boundary condition (\ref{1.8}) on a part of the boundary
has not been studied in literature yet. This is mainly because one
at first needs the information on regularity of solutions of the
corresponding steady Stokes problem, and there are only two papers
which bring this information: 1) paper \cite{KuBe}, where the
authors studied a flow in a 2D channel $D$ of a special geometry,
considering the homogeneous Dirichlet boundary condition on the
walls and condition (\ref{1.8}) on the outflow, and proved that
the velocity is in $\bfW^{2-\beta,2}(D)$ for certain $\beta>0$
depending on the geometry of $D$, provided that
$\bff\in\bfL^2(D)$, and 2) paper \cite{TNe4}, where the inclusion
of the solution $(\bfu,p)$ of the Stokes problem
(\ref{1.1})--(\ref{1.8}) to $\bfW^{2,2}(\Omega)\times
W^{1,2}(\Omega)$ has been recently proven under natural
assumptions on $\bff$, $\bfg$ and $\bfh$.

In this context, note that one usually says that the Stokes
problem has the {\it maximum regularity property,} if the solution
$\bfu$, respectively $p$, has by two, respectively one, spatial
derivatives more than function $\bff$, integrable with the same
power as $\bff$.

In general, the maximum regularity property of solutions of the
steady Stokes problem is mostly known if domain $\Omega$ is
sufficiently smooth, see e.g.~\cite[Theorem I.2.2]{Te},
\cite[Theorem III.3]{La}, \cite[Theorem IV.6.1]{Ga} and
\cite[Theorem III.2.1.1]{So} for problems with inhomogeneous
Dirichlet boundary conditions, \cite{AmSe}, \cite{ChOsQi} for
problems with the Navier--type boundary condition and
\cite{AcAmCoGh}, \cite{EsGh}, \cite{ChQi} for problems with
Navier's boundary condition  on the whole boundary. Concerning
non--smooth domains, we can cite \cite{Gr2}, \cite{KeOs} and
\cite{Dau}, where the authors considered the Stokes problem in a
2D polygonal domain with the Dirichlet boundary conditions, and
the aforementioned paper \cite{TNe4}, where the maximum regularity
property of the Stokes problem (\ref{1.1})--(\ref{1.8}) has been
proven in the $L^2$--framework.

\vspace{4pt} \noindent
{\bf On the results of this paper.} \ The main purpose of this
paper is to generalize the results from \cite{TNe4} from the
$L^2$--framework to the general $L^r$--framework for
$r\in(1,\infty)$. We use the results from \cite{TNe5}, where the
existence of a weak solution $\bfu\in\bfW^{1,r}(\Omega)$ to the
Stokes problem (\ref{1.1})--(\ref{1.5}), (\ref{1.8} is proven. It
is also shown in \cite{TNe5} that an associated pressure $p\in
L^r(\Omega)$ can be chosen so that $\bfu$ and $p$ satisfy
equations (\ref{1.1}), (\ref{1.2}) in the sense of distributions
in $\Omega$ and the boundary condition (\ref{1.8}) is satisfied as
an equality in $\bfW^{-1/r,r}(\Gammao)$. In this paper, we
consider smooth input data $\bff$, $\bfg$ and $\bfh$ and we prove
the existence of a strong solution $(\bfu,p)\in
\bfW^{2,r}(\Omega)\times W^{1,r}(\Omega)$ of the Stokes problem
(\ref{1.1})--(\ref{1.8}), see Theorem \ref{T2}. We also explain
how this theorem can be generalized so that it yields $(\bfu,p)\in
\bfW^{s+2,r}(\Omega)\times W^{s+1,r}(\Omega)$ for
$s\in\{0\}\cup\N$. These results do not follow from the previous
cited papers on the Stokes problem, because our domain $\Omega$ is
not smooth and we consider three different types of boundary
conditions on $\partial\Omega$. Two types of conditions ``meet''
in the corners $\Am$, $\Ap$, $\Bm$ and $\Bp$ of domain $\Omega$.
As auxiliary results of an independent importance, we present
Lemma \ref{L3} (on an appropriate extension of the velocity
profile $\bfg$ from $\Gammai$ to $\Omega$).

Finally, note that the presented results on the $L^r$--maximum
regularity property of the considered Stokes--type problem play a
fundamental role in studies of regularity and the structure of the
set of weak and strong solutions to the corresponding
Navier--Stokes problem. A paper on this theme is being prepared.

\section{Notation and auxiliary results} \label{S2}

{\bf Notation.} \ We assume that $1<r<\infty$ throughout the
paper.

\begin{list}{$\circ$}
{\setlength{\topsep 0.5mm}
\setlength{\itemsep 0.5mm}
\setlength{\leftmargin 14pt}
\setlength{\rightmargin 0pt}
\setlength{\labelwidth 6pt}}

\item
Recall that $\Omega$ is a Lipschitzian domain in $\R^2$, sketched
on Fig.~1. Its boundary consists of the curves $\Gammai$,
$\Gammao$, $\Gammam$, $\Gammap$ and $\Gammaw$, described in
Section \ref{S1}. We assume that the curves $\Gammap$, $\Gammam$
are of the class $C^{\infty}$ and $\Gammaw$ is of the class $C^2$.
We denote by $\bfn=(n_1,n_2)$ the outer normal vector field on
$\partial\Omega$.

\item
$\Gammai^0$, respectively $\Gammao^0$, denotes the open line
segment with the end points $\Am,\, \Ap$, respectively $\Bm,\,
\Bp$. Similarly, $\Gammam^0$, respectively $\Gammap^0$ denotes the
curve $\Gammam$, respectively $\Gammap$, without the end points
$\Am$, $\Bm$, respectively $\Ap$, $\Bp$.

\item
We denote by $\|\, .\, \|_r$ the norm in $L^r(\Omega)$ or in
$\bfL^r(\Omega)$ or in $L^r(\Omega)^{2\times 2}$. Similarly, $\|\,
.\, \|_{s,r}$ is the norm in $W^{s,r}(\Omega)$ or in
$\bfW^{s,r}(\Omega)$ or in $W^{s,r}(\Omega)^{2\times 2}$.

\item
Recall that $\cO=\R^2_{(0,d)}\smallsetminus
\cup_{k=-\infty}^{\infty}P_k$. For $k\in\N$, we denote by
$W^{k,r}_{\rm per}(\cO)$ the space of functions from
$W^{k,r}_{loc}(\cO)$, $\tau$--periodic in variable $x_2$.

\item
$W^{k,r}_{\rm per}(\Omega)$ is the space of functions, that can be
extended from $\Omega$ to $\cO$ as functions in $W^{k,r}_{\rm
per}(\cO)$. (The traces of these functions on $\Gammam$ and
$\Gammap$ satisfy the condition of periodicity, analogous to
(\ref{1.5}).)

\item
$W^{k-1/r,r}_{per}(\gammao)$ (for $k\in\N$) denotes the space of
$\tau$--periodic functions in $W^{k-1/r,r}_{loc}(\gammao)$.

\item
$W^{k-1/r,r}_{\rm per}(\Gammao)$ is the space of functions from
$W^{k-1/r,r}(\Gammao)$, that can be extended from $\Gammao$ to
$\gammao$ as functions in $W^{k-1/r,r}_{\rm per}(\gammao)$.

\item
Vector functions and spaces of vector functions are denoted by
boldface letters. Spaces of 2nd--order tensor functions are
denoted by the superscript $2\times 2$.

\item
$\cCs$ denotes the linear space of all infinitely differentiable
divergence--free vector functions in $\overline{\Omega}$, whose
support is disjoint with $\Gammai\cup\Gammaw$ and that satisfy,
together with all their derivatives (of all orders), the condition
of periodicity (\ref{1.5}). Note that each $\bfw\in\cCs$
automatically satisfies the outflow condition
$\int_{\Gammao}\bfw\cdot\bfn\; \rmd l=0$.

\item
$\Vs{r}$ is the closure of $\cCs$ in $\bfW^{1,r}(\Omega)$. It is a
space of divergence--free vector functions from
$\bfW^{1,r}(\Omega)$, whose traces on $\Gammai\cup\Gammaw$ are
equal to zero and the traces on $\Gammam$ and $\Gammap$ satisfy
the condition of periodicity (\ref{1.5}). Since functions from
$\Vs{r}$ are equal to zero on $\Gammai\cup\Gammaw$ (in the sense
of traces) and domain $\Omega$ is bounded, the norm in $\Vs{r}$ is
equivalent to $\|\nabla.\, \|_r$.

\item
The conjugate exponent to $r$ is denoted by $r'$, the dual space
to $\bfW^{1,r'}_0(\Omega)$ is denoted by $\bfW^{-1,r}(\Omega)$ and
the dual space to $\bfW^{1,r'}(\Omega)$ is denoted by
$\bfW^{-1,r}_0(\Omega)$. The corresponding norms are denoted by
$\|\, .\, \|_{\bfW^{-1,r}}$ and $\|\, .\, \|_{\bfW^{-1,r}_0}$,
respectively.

\item
$\Vsd{r}$ denotes the dual space to $\Vs{r'}$. The duality pairing
between $\Vsd{r}$ and $\Vs{r'}$ is denoted by $\langle\, .\, ,\,
.\, \rangle_{(\vsd{r},\vs{r'})}$. The norm in $\Vsd{r}$ is denoted
by $\|\, .\, \|_{\vsd{r}}$.

\item
Denote by $\cA_r$ the linear mapping of $\Vs{r}$ to $\Vsd{r}$,
defined by the equation
\begin{equation}
\blangle\cA_r\bfv,\bfw\brangle_{(\vsd{r},\vs{r'})}\ =\
(\nabla\bfv,\nabla\bfw) \qquad \mbox{for}\ \bfv\in\Vs{r}\
\mbox{and}\ \bfw\in\Vs{r'}, \label{2.1}
\end{equation}
where $(\nabla\bfv,\nabla\bfw)$ represents the integral
$\int_{\Omega}\nabla\bfv:\nabla\bfw\; \rmd\bfx$.

\item
$\R^2_{d-}$ denotes the half-plane $\{(x_1,x_2)\in\R^2;\ x_1<d\}$.

\item
We use $c$ as a generic constant, i.e.~a constant whose values may
change throughout the text.

\end{list}

Further, we cite some auxiliary results from previous papers. They
all concern in a certain sense the equation $\cA_r\bfv=\bff$,
which can be interpreted as the weak Stokes problem. The first
lemma comes from \cite[Theorem 1]{TNe5}:

\begin{lemma} \label{L1}
$\cA_r$ is a bounded, closed and one--to--one operator from
$\Vs{r}$ to $\Vsd{r}$ with $D(\cA_r)=\Vs{r}$ and
$R(\cA_r)=\Vsd{r}$. The adjoint operator is to $\cA_r$ is
$\cA_{r'}$.
\end{lemma}

\noindent
The next lemma follows from \cite[Theorem 2.5]{GeHeHi}.

%
%
\begin{lemma} \label{L2}
Let $\bff\in\bfW^{-1,r}_0(\Omega)$. Then there exists $\bbF\in
L^r(\Omega)^{2\times 2}$, satisfying $\div\bbF=\bff$ in the sense
of distributions in $\Omega$ and
\begin{equation}
\|\bbF\|_r\ \leq\ c\, \|\bff\|_{\bfW^{-1,r}_0}\, , \label{2.2}
\end{equation}
where $c$ is independent of $\bff$ and $\bbF$.
\end{lemma}

Define $\bfF\in\Vsd{r}$ by the formula
\begin{equation}
\blangle\bfF,\bfw\brangle_{(\vsd{r},\vs{r'})}\ :=\
-\int_{\Omega}\bbF:\nabla\bfw\; \rmd\bfx \label{2.3}
\end{equation}
for all $\bfw\in\Vs{r'}$. Obviously, $\|\bfF\|_{\vsd{r}}\leq c\,
\|\bbF\|_r$.

\begin{lemma} \label{L3}
Let $\bfg\in\bfW^{1-1/r,r}_{\rm per}(\Gammai)$ be a given function
on $\Gammai$. There exists a divergence--free extension
$\bfgs\in\bfW^{1,r}_{\rm per}(\Omega)$ of $\bfg$ from $\Gammai$ to
$\Omega$ and a constant $\cn05>0$, independent of $\bfg$, such
that

\vspace{4pt}
a) \ $\|\bfgs\|_{1,r}\leq\cc05\, \|\bfg\|_{1-1/r,r;\, \Gammai}$,

\vspace{4pt}
b) \ $\bfgs=(\Phi/\tau)\, \bfe_1$ \ in a neighbourhood of
$\Gammao$, where $\Phi=-\int_{\Gammai}\bfg\cdot\bfn\; \rmd l$.
\end{lemma}

\noindent
Lemma \ref{L3} is a slight modification of Lemma 2 in \cite{TNe5}
in the sense that the function $\bfg$ is supposed to be in
$\bfW^{1-1/r,r}_{\rm per}(\Gammai)$ instead of
$\bfW^{s,r}(\Gammai)$ (for $s>1/r$ if $1<r\leq 2$ and $s=1-1/r$ if
$r>2$), as in \cite{TNe5}. It can be proven by means of the same
arguments as Lemma 2 in \cite{TNe5}.

Define $\bfG\in\Vsd{r}$ by the formula
\begin{equation}
\blangle\bfG,\bfw\brangle_{(\vsd{r},\vs{r'})}\ :=\
\int_{\Omega}\nabla\bfg_*:\nabla\bfw\; \rmd\bfx. \label{2.4}
\end{equation}
The norm of $\bfG$ in $\Vsd{r}$ satisfies $\|\bfG\|_{\vsd{r}}\leq
c\, \|\nabla\bfgs\|_r$. Finally, the next Lemma \ref{L4} follows
from \cite[Theorem 2]{TNe5}.

\begin{lemma} \label{L4}
Let the elements $\bfF$ and $\bfG$ of $\Vsd{r}$ be defined by
formulas (\ref{2.3}) and (\ref{2.4}), respectively. Then there
exists a unique solution $\bfv\in\Vs{r}$ of the equation
$\nu\cA_r\bfv=\bfF+\nu\br\bfG$. Moreover, there exists an
associated pressure $p\in L^r(\Omega)$ such that
$\bfu:=\bfg_*+\bfv$ and $p$ satisfy the equation
\begin{equation}
-\nu\Delta\bfu+\nabla p\ =\ \div\bbF \label{2.5}
\end{equation}
in the sense of distributions in $\Omega$,
\begin{equation}
(-\nu\br\nabla\bfu-p\br\bbI-\bbF)\cdot\bfn\ =\ \bfzero
\label{2.6}
\end{equation}
holds as an equality in $\bfWp^{-1/r,r}(\Gammao)$ and
\phantom{$\cn02$}
\begin{equation}
\|\bfu\|_{1,r}+\|p\|_r\ \leq\ \cc02\, \bigl(
\|\bbF\|_r+\|\nabla\bfgs\|_r \bigr), \label{2.7}
\end{equation}
where $\cc02=\cc02(\Omega,\nu)$.
\end{lemma}

\section{The strong Stokes problem} \label{S3}

\begin{lemma} \label{L6}
Let $\bff\in\bfL^r(\Omega)$ and $\bfh\in\bfW^{1-1/r,r}_{\rm
per}(\Gammao)$ be given. Then there exists $\bbF\in W^{1,r}_{\rm
per} (\Omega)^{2\times 2}$, such that $\div\bbF=\bff$ a.e.~in
$\Omega$, $\bbF\cdot\bfn=\bfh$ a.e.~on $\Gammao$ in the sense of
traces and
\begin{equation}
\|\bbF\|_{1,r}\ \leq\ c\, \bigl( \|\bff\|_r+\|\bfh\|_{1-1/r,r;\,
\Gammao} \bigr), \label{3.1}
\end{equation}
where $c=c(\Omega,r)$.
\end{lemma}

\begin{proof}
Denote $\Omegah:=\Omega\cup P$. Then $|\Omegah|=\tau d$. Define
$\bff:=\bfzero$ in $P$. Thus, $\bff\in\bfL^r(\Omegah)$. Put
$\bfr:=|\Omegah|^{-1}\, \int_{\Omegah}\bff\; \rmd\bfx$.

Let $\zeta=\zeta(x_1)$ be a smooth real function in $[0,d]$, such
that $\zeta(0)=1$ and $\zeta$ be supported in $[0,\delta]$, where
$\delta>0$ is so small that the profile $P$ (see Fig.~1) lies in
the stripe $\delta<x_1<d-\delta$. Since
$\int_{\Omegah}\zeta'(x_1)\, \rmd\bfx=-\tau$, we have
$\int_{\Omegah}(\bff+\bfr d\, \zeta'(x_1))\; \rmd\bfx=\bfzero$.
Thus, due to \cite[Theorem III.3.3]{Ga}, there exists $\bbF_0\in
W^{1,r}_0(\Omegah)^{2\times 2}$, such that $\div\bbF_0=\bff+\bfr
d\, \zeta'$ a.e.~in $\Omegah$ and
\begin{equation}
\|\bbF_0\|_{1,r;\, \Omegah}\ \leq\ c\, \|\bff+\bfr d\,
\zeta'\|_{r;\, \Omegah}\ \leq\ c\, \|\bff\|_r,
\label{3.2}
\end{equation}
where $c=c(\tau,d,\zeta)$. Since $\bfr d\, \zeta'(x_1)=\div[\bfr
d\, \zeta(x_1)\otimes \bfe_1]$, $\bbF_0$ also satisfies
$\div[\bbF_0-\bfr d\, \zeta(x_1) \\ \otimes\bfe_1]=\bff$ a.e.~in
$\Omega$.

We will further construct the tensor function $\bbF$ in the form
\begin{equation}
\bbF\ =\ \bbF_0-\bfr d\, \zeta(x_1)\otimes\bfe_1+\bbH_1+\bbH_0,
\label{3.3}
\end{equation}
where $\bbH_1$ and $\bbH_0$ are defined as follows:

\vspace{4pt}
{\it 1) Function $\bbH_1$.} \ Put
$\overline{\bfh}\equiv(\overline{h}_1,\overline{h}_2):=
\tau^{-1}\int_{b_0}^{b_1}\bfh(d,\vartheta)\, \rmd\vartheta$. Let
$\bbH_1$ be the $2\times 2$ matrix, whose 1st column is $(0,0)^T$
and the second column is $(\overline{h}_1,\overline{h}_2)^T$. Then
$\div\bbH_1=\bfzero$ and $\bbH_1\cdot\bfe_1=\overline{\bfh}$ on
$\Gammao$.

\vspace{4pt}
{\it 2) Function $\bbH_0$.} \ Put
$\bfh_0\equiv(h_{01},h_{02}):=\bfh-\overline{\bfh}$. Then
$\int_{\Gammao}\bfh_0\; \rmd l=0$.

The tensor function $\bbH_0$ is constructed so that its $i$--th
row (for $i=1,2$) has the form $\bbH_{0i}= \nabla^{\perp} \psi_i$,
where $\nabla^{\perp}= (\partial_2,-\partial_1)$, $\psi_i\in
W^{2,r}(\Omega)$ and
\begin{align}
& \nabla^{\perp}\psi_i\cdot\bfn\ =\ h_{0i} && \mbox{a.e.~on}\
\Gammao, \label{3.4} \\
& \nabla^{\perp}\psi_i(x_1,x_2+\tau)\ =\ \nabla^{\perp}
\psi_i(x_1,x_2) && \mbox{for a.a.}\ (x_1,x_2)\in\Gammam.
\label{3.5}
\end{align}
As $\bfn=\bfe_1$ on $\Gammao$, condition (\ref{3.4}) means that
\begin{equation}
\partial_2\psi_i\ =\ h_{0i} \qquad \mbox{a.e.~on}\ \Gammao.
\label{3.6}
\end{equation}
Condition (\ref{3.5}) can be rewritten in the form
\begin{equation}
\nabla\psi_i(x_1,x_2+\tau)\ =\ \nabla\psi_i(x_1,x_2) \qquad
\mbox{for a.a.}\ (x_1,x_2)\in\Gammam. \label{3.7}
\end{equation}
As the function $h_{0i}$ lies in $W^{1-1/r,r}_{\rm per}(\Gammao)$,
it can be extended as a $\tau$--periodic function to the whole
straight line $\gammao$, so that the extended function (which we
again denote by $h_{0i}$) is in $W^{1-1/r,r}_{loc}(\gammao)$.
Moreover, the integral of $h_{0i}$ on any line segment on
$\gammao$ of length $\tau$ is zero. Put
$\psi_i(d,x_2):=\int_0^{x_2} h_{0i}(d,\vartheta)\; \rmd\vartheta$.
Function $\psi_i(d,\, .\, )$ is in $W^{2-1/r,r}_{loc}(\gammao)$,
satisfies condition (\ref{3.6}) and is $\tau$--periodic. Applying
\cite[Theorem II.4.4]{Ga}, we deduce that there exists an
extension of $\psi_i$ from $\gammao$ to the half-plane $\R^2_{d-}$
(which we again denote by $\psi_i$), such that $\psi_i\in
W^{2,r}_{loc}(\R^2_{d-})$, $\psi_i$ is supported in the stripe
$d-\delta\leq x_1\leq d$ and the restriction of $\psi_i$ to
$\Omega$ satisfies
\begin{equation}
\|\psi_i\|_{2,r}\ \leq\ c\, \|\psi_i\|_{2-1/r,r;\, \Gammao}\ \leq\
c\, \|h_{0i}\|_{1-1/r,r;\, \Gammao}. \label{3.8}
\end{equation}
As $\psi_i(d,x_2)$ is $\tau$--periodic in variable $x_2$, the
extension can also be constructed so that it is $\tau$--periodic
in variable $x_2$. Thus, the functions $\psi_i$ (for $i=1,2$)
satisfy (\ref{3.7}). The identities
$\bbH_{0i}=\nabla^{\perp}\psi_i$ ($i=1,2$) guarantee that
$\div\bbH_{0i}=0$ ($i=1,2$), which yields $\div\bbH_0=\bfzero$.
Obviously, $\bbH_0\in W^{1,r}_{\rm per}(\Omega)^{2\times 2}$,
$\bbH_0\cdot\bfn=\bfh_0$ on $\Gammao$ and
\begin{equation}
\|\bbH_0\|_{1,r}\ \leq\ c\, \|\bfh_0\|_{1-1/r,r;\, \Gammao},
\label{3.10}
\end{equation}
where $c=c(\Omega)$.

Now, it follows from the construction of $\bbF_0$, $\bbH_1$,
$\bbH_0$ and formula (\ref{3.3}) that $\bbF\in W^{1,r}_{\rm
per}(\Omega)$, $\div\bbF=\bff$ a.e.~in $\Omega$ and
$\bbF\cdot\bfn=\bfh$ on $\Gammao$. The estimate (\ref{3.1})
essentially follows from (\ref{3.3}) and (\ref{3.10}).
\end{proof}

\begin{remark} \label{R2} \rm
Suppose that function $\bfg$ in the assumptions of Lemma \ref{L3}
is in $\bfW^{2-1/r,r}_{\rm per}(\Gammai)$. Then the extension
$\bfg_*$ of $\bfg$ from $\Gammai$ to $\Omega$ can be constructed
so, that, in addition to the properties a), b), named in Lemma
\ref{L3}, it is in $\bfW^{2,r}(\Omega)$ and

\vspace{4pt}
c) \ $\|\bfg_*\|_{2,r}\leq c\, \|\bfg\|_{2-1/r,r;\, \Gammai}$, \
where $c$ is independent of $\bfg$,

\vspace{4pt}
d) \ $\bfg_*$ satisfies the condition of periodicity (\ref{1.6})
on $\Gammam\cup\Gammap$. \rm

\vspace{4pt} \noindent
The possibility of the construction of $\bfg_*$ with the
properties c) and d) follows, similarly as Lemma \ref{L3}, from an
appropriate modification of the proof of Lemma 2 in \cite{TNe5}.
Using the higher regularity of $\bfg$, one can in principle apply
the same arguments so that one obtains the extension $\bfg_*$ with
all the properties a) -- d).
\end{remark}

\begin{theorem}[on a strong solution of the Stokes problem
(\ref{1.1})--(\ref{1.8})] \label{T2} Let the \\ closed curve
$\Gammaw$ (which is the boundary of the profile) be of the class
$C^2$, $\bff\in\bfL^r(\Omega)$, $\bfh\in\bfW^{1-1/r,r}_{\rm
per}(\Gammao)$, $\bfg\in\bfW^{2-1/r,r}_{\rm per}(\Gammai)$ be
given. Let $\bbF$ and $\bfgs$ be the functions, given by Lemma
\ref{L6}, Lemma \ref{L3} and Remark \ref{R2}. Let the functionals
$\bfF$ and $\bfG$ be defined by formulas (\ref{2.3}) and
(\ref{2.4}), respectively. Then

\begin{list}{}
{\setlength{\topsep 0pt}
\setlength{\itemsep 1pt}
\setlength{\leftmargin 15pt}
\setlength{\rightmargin 0pt}
\setlength{\labelwidth 10pt}}

\item[1)]
the unique solution $\bfv$ of the equation
$\nu\cA_r\bfv=\bfF+\nu\br\bfG$ belongs to
$\Vs{r}\cap\bfW^{2,r}(\Omega)$,

\item[2)]
there exists an associated pressure $p\in W^{1,r}(\Omega)$ so that
the functions $\bfu:=\bfg_*+\bfv$ and $p$ satisfy equations
(\ref{1.1}) (with $\bff=\div\bbF$) and (\ref{1.2}) a.e.~in
$\Omega$,

\item[3)]
$\bfu$, $p$ satisfy boundary conditions (\ref{1.3}), (\ref{1.4})
and (\ref{1.8}) in the sense of traces on $\Gammai$, $\Gammaw$ and
$\Gammao$, respectively,

\item[4)]
$\bfu$, $p$ satisfy the conditions of periodicity
(\ref{1.5})--(\ref{1.7}) in the sense of traces on $\Gammam$ and
$\Gammap$,

\item[5)]
there exists a constant $\cn01=\cc01(\nu,\Omega)$, such that
\begin{equation}
\|\bfu\|_{2,r}+\|\nabla p\|_r\ \leq\ \cc01\, \bigl( \|\bff\|_r+
\|\bfg\|_{2-1/r,r;\, \Gammai}+\|\bfh\|_{1-1/r,r;\, \Gammao}
\bigr). \label{3.11}
\end{equation}
\end{list}
\end{theorem}

\begin{remark} \label{R3} \rm
The conclusions $\bfu\in\bfW^{2,r}(\Omega)$ and $p\in
W^{1,r}(\Omega)$, following from Theorem \ref{T2}, together with
inequality (\ref{3.11}), represent the maximum regularity property
of the Stokes problem (\ref{1.1})--(\ref{1.8}).
\end{remark}

\noindent
{\bf Proof of Theorem \ref{T2}.} \ The existence and uniqueness of
the solution $\bfv\in\Vs{r}$ of the equation
$\nu\cA_r\bfv=\bfF+\nu\br\bfG$ and an associated pressure $p\in
L^r(\Omega)$ are guaranteed by Lemma \ref{L4}. It also follows
from Lemma 4 that the functions $\bfu:=\bfgs+\bfv$ and $p$ satisfy
(\ref{2.5})--(\ref{2.7}).

\vspace{4pt} \noindent
{\bf Assume at first that $r\geq 2$.} \ Then, due to \cite[Theorem
2]{TNe4}, $\bfv\in\bfW^{2,2}_{\rm per}(\Omega)$ and $p\in
W^{1,2}_{\rm per}$, $\bfu$ and $p$ satisfy equations (\ref{2.5}),
(\ref{1.2}) a.e.~in $\Omega$ and the boundary conditions
(\ref{1.3}), (\ref{1.4}), (\ref{1.8}) in the sense of traces on
$\Gammai$, $\Gammaw$ and $\Gammao$, respectively. Moreover,
\begin{equation}
\|\bfu\|_{2,2}+\|\nabla p\|_2\ \leq\ c\, \bigl(
\|\bff\|_2+\|\bfg\|_{3/2,2;\, \Gammai}+\|\bfh\|_{1/2,2;\, \Gammao}
\bigr). \label{3.12}
\end{equation}
This implies the validity of statements 2) and 3). The validity of
statement 4) follows from the fact that $\bfv\in\bfW^{2,2}_{\rm
per}(\Omega)$ and the extended function $\bfgs$ satisfies the
conditions of periodicity (\ref{1.5}), (\ref{1.6}). Thus, we only
need to prove items 1) and 5).

We split the proof to three lemmas, where we successively show
that $\bfv\in\bfW^{2,r}$ and $p\in W^{1,r}$ in the interior of
$\Omega$ plus the neighbourhood of $\Gammaw$ and the neighborhood
of any closed subset of $\Gammai^0$ (Lemma \ref{L7}), in the
neighborhood of $\Gammao^0$ (Lemma \ref{L8}) and in the
neighborhoods of $\Gammam$ and $\Gammap$ (Lemma \ref{L9}). Lemmas
\ref{L7}--\ref{L9} also provide estimates, which finally imply
(\ref{3.11}).

\begin{lemma} \label{L7}
Let $\Omega'$ be a sub-domain of $\Omega$, such that
$\overline{\Omega'}\subset\Omega\cup\Gammai^0\cup\Gammaw$. Then
$\bfv\in\bfW^{2,r}(\Omega')$, $p\in W^{1,r}(\Omega')$ and
\begin{equation}
\|\bfv\|_{2,r;\, \Omega'}+\|\nabla p\|_{r;\, \Omega'}\ \leq\ c\,
\bigl( \|\div\bbF\|_r+\|\bfg_*\|_{2,r}+\|\bfv\|_{1,r} \bigr),
\label{3.13}
\end{equation}
where $c=c(\nu,\Omega,\Omega')$.
\end{lemma}

\begin{proof}
Let $\Omega''$ be a smooth sub-domain of $\Omega$ such that
$\Omega'\subset\Omega''\subset\Omega$,
$\overline{\Omega''}\subset\Omega\cup\Gammai^0\cup\Gammaw$ and
$\dist(\partial\Omega''\cap\Omega;\, \partial\Omega'\cap
\Omega)>0$. Let $\eta$ be an infinitely differentiable cut--off
function in $\Omega$ such that $\supp\eta\subset
\overline{\Omega''}$ and $\eta=1$ in $\Omega'$. Put
$\bfvt:=\eta\bfv$ and $\pt:=\eta\br p$. The functions $\bfvt$,
$\pt$ represent a strong solution of the problem
\begin{align}
-\nu\Delta\bfvt+\nabla\pt\ &=\ \bfft && \mbox{in}\ \Omega'',
\label{3.14} \\
\div\bfvt\ &=\ \hht && \mbox{in}\ \Omega'', \label{3.15} \\
\bfvt\ &=\ \bfzero && \mbox{on}\ \partial\Omega'', \label{3.16}
\end{align}
where
\begin{displaymath}
\bfft:=\eta\, \div\bbF-2\nu\br\nabla\eta\cdot\nabla\bfv-\nu\,
(\Delta\eta)\, \bfv-(\nabla\eta)\, p+\nu\br\eta\br\Delta\bfg_*
\qquad \mbox{and} \qquad \hht:=\nabla\eta\cdot\bfv.
\end{displaymath}
As $\div\bbF\in\bfL^r(\Omega)$ and $\bfu\in\bfW^{1,r}(\Omega)$,
$\bfv\in\Vs{r}$ and $p\in L^r(\Omega)$ (satisfying (\ref{2.7}), we
have $\bfft\in\bfL^r(\Omega)$, $\hht\in W^{1,r}(\Omega)$ and
\begin{alignat}{3}
& \|\bfft\|_r\ &&\leq\ c\, \bigl( \|\div\bbF\|_r+
\|\bfg_*\|_{2,r}+\|\bfv\|_{1,r} \bigr), \label{3.17} \\
& \|\hht\|_{1,r}\ &&\leq\ c\, \|\bfv\|_r\ \leq\ c\,
\|\bfv\|_{1,r}, \label{3.18}
\end{alignat}
where $c=c(\nu,\eta)$. Due to \cite[Proposition I.2.3, p.~35]{Te},
$\bfvt\in\bfW^{2,r}(\Omega'')$, $\pt\in W^{1,r}(\Omega''$ and
\begin{displaymath}
\|\bfvt\|_{2,r;\, \Omega''}+\|\nabla\pt\|_{r;\, \Omega''}\ \leq\
c\, \bigl( \|\bfft\|_{r;\, \Omega''}+\|\hht\|_{1,r;\, \Omega''}
\bigr),
\end{displaymath}
where $c=c(\Omega'')$. Consequently, $\bfv\in\bfW^{2,r}(\Omega')$,
$p\in W^{1,r}(\Omega')$ and (\ref{3.13}) holds.
\end{proof}

\vspace{4pt}
Recall that $\Gammao^0$ is the open line segment with the end
points  $\Bm$ and $\Bp$.

\begin{lemma} \label{L8}
Let $\Omega'$ be a sub-domain of $\Omega$, such that
$\overline{\Omega'}\subset\Omega\cup\Gammao^0$. Then
$\bfv\in\bfW^{2,r}(\Omega')$, $p\in W^{1,r}(\Omega')$ and the
inequality (\ref{3.13}) holds.
\end{lemma}

\begin{proof}
Here, we must use a different method than in the proof of Lemma
\ref{L7}. The reason is that we cannot apply Proposition I.2.3
from \cite{Te}, because it concerns the Stokes problem with the
Dirichlet boundary condition, which we do not have on $\Gammao$.

Denote by $\Gammao'$ the intersection of $\overline{\Omega'}$ with
$\Gammao^0$. We may assume, without loss of generality, that
$\Gammao'\not=\emptyset$ and it is a line segment.

Let $\rho_2>\rho_1>0$. Denote $U_1:=\{\bfx\in\R^2;\
\dist(\bfx,\Gammao')<\rho_1\}$ and $U_2:=\{\bfx\in\R^2;\
\dist(\bfx,\Gammao')<\rho_2\}$. Suppose that $\rho_2$ is so small
that $\overline{U_2}\cap\R^2_{d-}\subset\Omega$. (See Fig.~2.)

\vspace{4pt} \noindent
{\it Step 1.} We will construct a divergence--free function
$\bfvt$ in $\bfW^{1,r}_0(U_2)$ that coincides with $\bfv$ in
$U_1\cap\R^2_{d-}$.

\begin{wrapfigure}[15]{r}{70mm}

\begin{center}
  \setlength{\unitlength}{0.4mm}
  \begin{picture}(162,157)
  \put(8,90){\vector(1,0){137}}
  \put(132,94){\small $x_1$}
  %
  \put(77.3,50){\line(0,1){90}}
  \qbezier(77,50)(47,64)(14,68)
  \qbezier(77,140)(47,154)(14,158)
  \dashline[+30]{2.2}(77.3,32)(77.3,164)
  \thicklines
  \put(77.3,75){\line(0,1){35}}
  \thinlines
  \qbezier[16](78,100)(89,102)(100,104)
  \put(102,102){\small $\Gammao'$}
  %
  \qbezier(77.3,124)(83.096,124)(87.198,119.898) 
  \qbezier(87.198,119.898)(91.3,115.796)(91.3,110) 
  \qbezier(91.3,75)(91.3,69.204)(87.198,65.102) 
  \qbezier(87.198,65.102)(83.096,61)(77.3,61) 
  \qbezier(77.3,124)(71.504,124)(67.402,119.898) 
  \qbezier(67.402,119.898)(63.3,115.796)(63.3,110) 
  \qbezier(63.3,75)(63.3,69.204)(67.402,65.102) 
  \qbezier(67.402,65.102)(71.504,61)(77.3,61) 
  \put(63.4,110){\line(0,-1){35}} \put(91.4,110){\line(0,-1){35}}
  \qbezier(77.3,118)(80.612,118)(82.956,115.656)
  \qbezier(82.956,115.656)(85.3,113.312)(85.3,110)
  \qbezier(85.3,75)(85.3,71.688)(82.956,69.344)
  \qbezier(82.956,69.344)(80.612,67)(77.3,67)
  \qbezier(77.3,118)(73.988,118)(71.644,115.656)
  \qbezier(71.644,115.656)(69.3,113.312)(69.3,110)
  \qbezier(69.3,75)(69.3,71.688)(71.644,69.344)
  \qbezier(71.644,69.344)(73.988,68)(77.3,67)
  \put(69.4,110){\line(0,-1){35}} \put(85.4,110){\line(0,-1){35}}
  \qbezier[11](85.4,80)(93.1,80)(100,80)
  \put(102,78){\small $U_1$}
  %
  \qbezier(77,110)(65,110)(55,120)
  \qbezier(55,120)(25,150)(25,100)
  \qbezier(25,100)(25,72)(60,72)
  \qbezier(60,72)(70,72)(77,75)
  \put(37,111){\small $\Omega'$}
  \put(61.5,131){\small $\Gammao$}
  \put(40,141.7){\small $\Gammap$} \put(40,51.7){\small $\Gammam$}
  \put(80,46){\small $B_0$} \put(80,137){\small $B_1$}
  \put(80,157){\small $\gammao$ (\small $x_1=d$)}
  \put(92,62){\small $U_2$}
  \put(8,10){Fig.~2: \ The sets $\Omega'$, $\Gammao'$, $U_1$ and $U_2$}
  \end{picture}
\end{center}
\end{wrapfigure}

Let $\eta$ be a $C^{\infty}$--function in $\R^2$, supported in
$\overline{U_2}$, such that $\eta=1$ in $U_1$ and $\eta$ is
symmetric with respect to the line $x_1=d$. (It means that
$\eta(d+\vartheta,x_2)=\eta(d-\vartheta,x_2)$ for all
$\vartheta,\, x_2\in\R$.)

Due to \cite{KMPT}, there exists a divergence--free extension
$\bfv_2$ of function $\bfv$ from $U_2\cap\R^2_{d-}$ to the whole
set $U_2$, such that $\bfv_2\in\bfW^{1,r}(U_2)$ and $\| \bfv_2
\|_{1,r;\, U_2}\leq c\, \|\bfv\|_{1,r}$, where $c$ is independent
of $\bfv$.

Since $\nabla\eta\cdot\bfv_2\in W^{1,r}_0(U_2)$ and
$\int_{U_2}\nabla\eta\cdot\bfv_2\; \rmd\bfx=0$, there exists (by
\cite[Theorem III.3.3]{Ga}) $\bfv_*\in\bfW^{2,r}_0(U_2)$, such
that $\div\bfv_*=\nabla\eta\cdot\bfv_2$ in $U_2$ and
\begin{align*}
\|\bfv_*\|_{2,r;\, U_2}\ &\leq\ c\,
\|\nabla\eta\cdot\bfv_2\|_{1,r;\, U_2} \\
&\leq\ c\, \|\bfv_2\|_{1,r;\, U_2}\ \leq\ c\, \|\bfv\|_{1,r},
\end{align*}
where $c$ is independent of $\bfv$. Extending $\bfv_*$ by zero to
$\R^2\smallsetminus U_2$, we have $\|\bfv_*\|_{2,r}\leq c\,
\|\bfv\|_{1,r}$. Put
\begin{equation}
\bfvt\ :=\ \eta\bfv_2-\bfv_*, \qquad \pt\ :=\ \eta\br p.
\label{3.20}
\end{equation}
Function $\bfvt$ is divergence--free, belongs to
$\bfW^{1,r}_0(U_2)$ and satisfies the estimates
\begin{displaymath}
\|\bfvt\|_{1,r;\, U_2}\ \leq\ c\, \bigl( \|\bfv_2\|_{1,r;\,
U_2}+\|\bfv_*\|_{1,r;\, U_2} \bigr)\ \leq\ c\, \|\bfv_2\|_{1,r;\,
U_2}\ \leq\ c\, \|\bfv\|_{1,r},
\end{displaymath}
where $c$ is independent of $\bfv$. The functions $\bfvt$, $\pt$
(defined by (\ref{3.20}) in $U_2\cap\R^2_{d-}$ and extended by
zero to $\R^2_{d-}$) satisfy the equation (\ref{3.14}) a.e.~in the
half-plane $\R^2_{d-}$, where function $\bfft$ now satisfies
\begin{displaymath}
\bfft\ :=\ \eta\, \div\bbF-2\nu\br\nabla\eta\cdot\nabla\bfv-\nu\,
(\Delta\eta)\, \bfv-(\nabla\eta)\, p+\nu\br\eta\Delta
\bfg_*+\nu\Delta\bfv_* \qquad \mbox{in}\ U_2\cap\R^2_{d-}
\end{displaymath}
and $\bfft:=\bfzero$ in $\R^2_{d-}\smallsetminus U_2$. This
function, although it is different from function $\bfft$ from the
proof of Lemma \ref{L7}, satisfies the estimate (\ref{3.17}).

Note that
\begin{align}
\nu\, \frac{\partial\bfvt}{\partial\bfn}-\pt\br\bfn\ &=\ \nu\eta\,
\frac{\partial\bfv}{\partial\bfn}-\nu\,
\frac{\partial\bfv_*}{\partial\bfn}-\eta\br p\br\bfn\ =\ \eta\,
\Bigl( \nu\, \frac{\partial\bfu}{\partial\bfn}-p\br\bfn\Bigr)-
\nu\, \frac{\partial\bfv_*}{\partial\bfn} \nonumber \\
&=\ -\eta\, \bbF\cdot\bfn-\nu\,
\frac{\partial\bfv_*}{\partial\bfn}\ =\ -\bfht \label{3.19}
\end{align}
a.e.~on $\Gammao$. We have used the identities $\bfu=\bfv+\bfgs$
(in $\Omega)$ and $\partial\bfgs/\partial\bfn=\bfzero$ (a.e.~on
$\Gammao$.

\vspace{4pt} \noindent
{\it Step 2.} Define
\begin{displaymath}
\bfht\ :=\ \eta\, \bbF\cdot\bfn+\nu\,
\frac{\partial\bfv_*}{\partial\bfn} \qquad (\mbox{in the sense of
traces on}\ \Gammao).
\end{displaymath}
Function $\bfht$ satisfies the estimates
\begin{align}
\|\bfht\|_{1-1/r,r;\, \Gammao}\ &\leq\ c\, \|\bbF\|_{1-1/r,r;\,
\Gammao}+\Bigl\| \frac{\partial\bfv_*}{\partial\bfn}
\Bigr\|_{1-1/r,r;\, \Gammao}\ \leq\ c\, \|\bbF\|_{1,r}+c\,
\|\bfv_*\|_{2,r;\, U_2} \nonumber \\
&\leq\ c\, \|\bbF\|_{1,r}+c\, \|\bfv\|_{1,r}. \label{3.21}
\end{align}
Let $\bbFt$ be a function in $W^{1,r}_{\rm per}(\Omega)^{2\times
2}$, provided by Lemma \ref{L6}, where we consider $\bfft$ instead
of $\bff$ and $\bfht$ instead of $\bfh$. Then $\div\bbFt=\bfft$
a.e.~in $\Omega$ and $\bbFt\cdot\bfn=\bfht$ a.e.~on $\Gammao$.
Moreover, due to (\ref{3.1}), (\ref{3.17}) and (\ref{3.21}),
\begin{equation}
\|\bbFt\|_{1,r}\ \leq\ c\, \|\bfft\|_r+c\, \|\bfht\|_{1-1/r,r;\,
\Gammao}\ \leq\ c \bigl(
\|\bbF\|_{1,r}+\|\bfgs\|_{2,r}+\|\bfv\|_{1,r} \bigr). \label{3.22}
\end{equation}
Let the functional $\bfFt\in\Vsd{r}$ be defined by the same
formula as (\ref{2.3}), where we only consider $\bbFt$ instead of
$\bbF$. We claim that $\nu\cA_r\bfvt=\bfFt$. Indeed, applying
(\ref{3.19}), we obtain for any $\bfw\in\Vs{r'}$:
\begin{align*}
\nu\, \langle\cA_r\bfvt & ,\bfw\rangle_{(\vsd{r},\vs{r'})}\ =\
\nu\int_{\Omega} \nabla\bfvt:\nabla\bfw\; \rmd\bfx\ =\
\int_{\Gammao}\nu\, \frac{\partial\bfvt}{\partial\bfn}\cdot\bfw\;
\rmd l-\nu\int_{\Omega}\Delta\bfvt\cdot\bfw\; \rmd\bfx \\
&=\ \int_{\Gammao}\nu\, \frac{\partial\bfvt}{\partial\bfn}\cdot
\bfw\; \rmd l+\int_{\Omega}(-\nabla\pt+\bfft)\cdot\bfw\; \rmd\bfx \\
&=\ \int_{\Gammao} \Bigl[ \nu\,
\frac{\partial\bfvt}{\partial\bfn}-\pt\br\bfn \Bigr]\cdot\bfw\;
\rmd l+\int_{\Omega}\div\bbFt\cdot\bfw\; \rmd\bfx \\
&=\ -\int_{\Gammao} \bfht\cdot\bfw\; \rmd l+
\int_{\Gammao}(\bbFt\cdot\bfn)\cdot\bfw\; \rmd l-\int_{\Omega}
\bbFt:\nabla\bfw\; \rmd\bfx \\
&=\ -\int_{\Omega} \bbFt:\nabla\bfw\; \rmd\bfx\ =\
\langle\bfFt,\bfw\rangle_{(\vsd{r},\vs{r'})}.
\end{align*}

\noindent
{\it Step 3.} In this part, we apply the method of difference
quotients (see \cite{Ag}, \cite{Gr1} and \cite{So}) in order to
derive the estimate (\ref{3.13}).

Recall that $\bfft$ is defined in $\R^2_{d-}$ and supported in the
closure of $U_2\cap\R^2_{d-}$, and $\bfht$ is defined in $\Gammao$
and supported in the closure of $U_2\cap\gammao$. For
$\delta\in\R$, denote
\begin{displaymath}
D_2^{\delta}\bfft(x_1,x_2):=\frac{\bfft(x_1,x_2+\delta)-
\bfft(x_1,x_2)}{\delta} \quad \mbox{and} \quad
D_2^{\delta}\bfht(d,x_2):=\frac{\bfht(d,x_2+\delta)-
\bbFt(d,x_2)}{\delta}.
\end{displaymath}
$D_2^{\delta}\bfft$ and $D_2^{\delta}\bfht$ are the so called {\it
difference quotients.}

As $\bbFt\in W^{1,r}_{\rm per}(\Omega)^{2\times 2}$, it can be
extended from $\Omega$ to $\Omega\cup P$ so that the extended
function is in $W^{1,r}_{\rm per}(\Omega\cup P)^{2\times 2}$,
$\div\bbFt=\bbO$ in $P$ and $\|\bbFt\|_{1,r;\, \Omega\cup P}\leq
c\, \|\bbFt\|_{1,r}$. Furthermore, $\bbFt$ can be extended from
$\Omega\cup P$ to the stripe
$\R^2_{(0,d)}:=\{\bfx=(x_1,x_2)\in\R^2;\ 0<x_1<d\}$ as
$\tau$--periodic function in variable $x_2$, lying in
$W^{1,r}_{loc}(\R^2_{(0,d)})$. Let us denote the extension again
by $\bbFt$ and define
\begin{displaymath}
D_2^{\delta}\bbFt(x_1,x_2)\ :=\ \frac{\bbFt(x_1,x_2+\delta)-
\bbFt(x_1,x_2)}{\delta}.
\end{displaymath}
Denote $\bbFt_{\delta}(x_1,x_2):=\delta^{-1}
\int_0^{\delta}\bbFt(x_1,x_2+ \vartheta)\; \rmd\vartheta$. Then
\begin{displaymath}
D_2^{\delta}\bbFt(x_1,x_2)\ =\
\frac{1}{\delta}\int_0^{\delta}\partial_2\bbFt(x_1,x_2+\vartheta)\;
\rmd\vartheta\ =\ \partial_2\bbFt_{\delta}(x_1,x_2).
\end{displaymath}
Using the $\tau$--periodicity of the function $\bbFt_{\delta}$ in
variable $x_2$ in $\R^2_{(0,d)}$ and applying H\"older's
inequality, we get
\begin{align*}
\|\bbFt_{\delta} & \|_r^r\ =\ \int_{\Omega}\biggl|\frac{1}{\delta}
\int_0^{\delta} \bbFt(x_1,x_2+\vartheta)\; \rmd
\vartheta\biggr|^r\; \rmd\bfx\ =\ \int_0^d\int_0^{\tau}
\biggl|\frac{1}{\delta} \int_0^{\delta} \bbFt(x_1,x_2+\vartheta)\;
\rmd \vartheta\biggr|^r2\; \rmd x_2\, \rmd x_1 \nonumber \\
&\leq\ \int_0^d\int_0^{\tau}\frac{1}{\delta}
\int_0^{\delta}\bigl|\bbFt(x_1,x_2+\vartheta)\bigr|^r\;
\rmd\vartheta\, \rmd x_2\, \rmd x_1\ =\ \int_0^d\frac{1}{\delta}
\int_0^{\delta}\int_0^{\tau}\bigl|\bbFt(x_1,y_2)\bigr|^r\;
\rmd y_2\, \rmd\vartheta\, \rmd x_1 \nonumber \\
&=\ \int_0^d \int_0^{\tau}\bigl|\bbFt(x_1,y_2)\bigr|^r\; \rmd
y_2\, \rmd x_1\ =\ \int_{\Omega}\bigl|\bbFt(\bfx)\bigr|^r\;
\rmd\bfx\ =\ \|\bbFt\|_r^r.
\end{align*}
We can similarly show that $\|\nabla\bbFt_{\delta}\|_r^r\leq
\|\nabla\bbFt\|_r^r$. Consequently, $\|\bbFt_{\delta}\|_{1,r}\leq
\|\bbFt\|_{1,r}$. Thus,
\begin{equation}
\| D_2^{\delta}\bbFt\|_r\ =\ \|\partial_2\bbFt_{\delta}\|_r\ \leq\
\|\bbFt_{\delta}\|_{1,r}\ \leq\ \|\bbFt\|_{1,r}\ \leq\ c\, \bigl(
\|\bfft\|_r+\|\bfht\|_{1-1/r,r;\, \Gammao}\bigr).
\label{2.36}
\end{equation}
Let $D_2^{\delta}\bfvt$ and $D_2^{\delta}\pt$ be defined by
analogy with $D_2^{\delta}\bfft$ and $D_2^{\delta}\bbFt$. The
functions $D_2^{\delta}\bfvt$, $D_2^{\delta}\pt$ satisfy the
equations
\begin{align*}
-\nu\Delta D_2^{\delta}\bfvt+\nabla D_2^{\delta}\pt\ &=\
\div D_2^{\delta}\bbFt, \\
\div D_2^{\delta}\bfvt\ &=\ 0
\end{align*}
a.e.~in $\Omega$. Since
\begin{displaymath}
\nu\, \frac{\partial\bfvt}{\partial\bfn}-\pt\br\bfn\ =\ \nu\,
\bigl( \eta\, \frac{\partial\bfv}{\partial\bfn}-\frac{\partial
\bfv_*}{\partial\bfn}\Bigr)-\eta\br p\br\bfn\ =\ -\eta\,
\bbF\cdot\bfn-\nu\, \frac{\partial\bfv_*}{\partial\bfn}\ =\ -\bfht
\end{displaymath}
on $\gammao$, $D_2^{\delta}\bfvt$ and $D_2^{\delta}\pt$ also
satisfy the boundary condition
\begin{displaymath}
-\nu\, \frac{\partial
D_2^{\delta}\bfvt}{\partial\bfn}+D_2^{\delta}\pt\br\bfn\ =\
D_2^{\delta}\bfht
\end{displaymath}
on $\Gammao$. From this, one can deduce that $\nu\cA_r
D_2^{\delta}\bfvt=\bfFt_{\delta}$, where the functional
$\bfFt_{\delta}$ in $\Vsd{r}$ is defined by the same formula as
(\ref{2.3}), where we only consider $D_2^{\delta}\bbFt$ instead of
$\bbF$. It follows from Lemma \ref{L1} that
\begin{displaymath}
\|\nabla D_2^{\delta}\bfvt\|_r\ \leq\
\|\bfFt_{\delta}\|_{\vsd{r}}.
\end{displaymath}
Since $\|\bfFt_{\delta}\|_{\vsd{r}}\leq\| D_2^{\delta}\bbFt\|_r
\leq c\, \bigl( \|\bfft\|_r+\|\bfht\|_{1-1/r,r;\, \Gammao}\bigr)$,
we obtain
\begin{equation}
\|\nabla D_2^{\delta}\bfvt\|_r\ \leq\ c\, \bigl(
\|\bfft\|_r+\|\bfht\|_{1-1/r,r;\, \Gammao}\bigr).
\label{2.37}
\end{equation}
Applying further Theorem 1 (with $\bfg_*=\bfzero$), (\ref{2.36})
and (\ref{2.37}), we obtain the estimate of $D_2^{\delta}\pt$:
\begin{equation}
\|D_2^{\delta}\pt\|_r\ \leq\ c\, \bigl( \|\nabla
D_2^{\delta}\bfvt\|_r+\| D_2^{\delta}\bbFt\|_r \bigr)\ \leq\ c\,
\bigl( \|\bfft\|_r+\|\bfht\|_{1-1/r,r;\, \Gammao}\bigr).
\label{2.38}
\end{equation}
As the right hand sides of (\ref{2.37}) and (\ref{2.38}) are
independent of $\delta$, we may let $\delta$ tend to $0$ and we
obtain
\begin{equation}
\|\nabla\partial_r\bfvt\|_2+\|\partial_2\pt\|_r\ \leq\ c\, \bigl(
\|\bfft\|_r+\|\bfht\|_{1-1/r,r;\, \Gammao}\bigr).
\label{2.39}
\end{equation}
This shows that $\partial_1\partial_2\vt_1$, $\partial_2^2\vt_1$,
$\partial_1\partial_2\vt_2$, $\partial_2^2\vt_2$ and
$\partial_2\pt$ are all in $L^r(\Omega)$ and their norms are less
than or equal to the right hand side of (\ref{2.39}).
Consequently, as $\bfvt$ is divergence--free, the same statement
also holds on $\partial_1^2\vt_1$. Now, from (\ref{3.13})
(considering just the first scalar component of this vectorial
equation), we deduce that $\partial_1\pt\in L^r(\Omega)$. Finally,
considering the second scalar component in equation (\ref{3.13}),
we obtain $\partial_1^2\vt_2\in L^r(\Omega)$, too. Thus, applying
also (\ref{3.16}) and (\ref{3.17}), we obtain
\begin{displaymath}
\|\bfvt\|_{2,r}+\|\nabla\pt\|_r\ \leq\ c\, \bigl(
\|\bbF\|_{1,r}+\|\bfg_*\|_{2,r}+\|\bfv\|_{1,r} \bigr).
\end{displaymath}
This inequality, formulas (\ref{3.18}), the estimate of
$\|\bfv_*\|_{2,r}$ and the fact that $\eta=1$ on $U_1$, in
combination with Lemma \ref{L7}, yield (\ref{3.12}).
\end{proof}

\vspace{4pt} \noindent
The next corollary is an immediate consequence of Lemmas \ref{L7}
and \ref{L8}.

\begin{corollary} \label{C1}
Let $\Omega'$ be a sub-domain of $\, \Omega,\, $ such that $\,
\overline{\Omega'}\subset\Omega\cup\Gammai^0\cup\Gammaw\cup
\Gammao^0.\, $ Then $\bfv\in\bfW^{2,r}(\Omega')$, $\, p\in
W^{1,r}(\Omega')$ and estimate (\ref{3.12}) holds.
\end{corollary}

\begin{lemma} \label{L9}
Let $\Omega'$ be a sub-domain of $\Omega$, such that
$\partial\Omega'\cap\Gammaw=\emptyset$ and $\Gammap\subset
\partial\Omega'$. Then
$\bfv\in\bfW^{2,r}(\Omega')$, $p\in W^{1,r}(\Omega')$ and
\begin{equation}
\|\bfv\|_{2,r;\, \Omega'}+\|\nabla p\|_{r;\, \Omega'}\ \leq\ c\,
\bigl( \|\div\bbF\|_r+\|\bfg_*\|_{2,r}+\|\bfv\|_{1,r} \bigr),
\label{3.26}
\end{equation}
where $c=c(\nu,\Omega,\Omega')$.
\end{lemma}

\begin{proof}
Consider $\delta>0$ and denote
\begin{align*}
\Am^{\delta} &:= \Am+\delta\br\bfe_2, & \Ap^{\delta} &=
\Ap+\delta\br\bfe_2, & \Bm^{\delta} &:=\Bm+\delta\br\bfe_2,
& \Bp^{\delta} &= \Bp+\delta\br\bfe_2, \\
\Gammai^{\delta} &:= \Gammai+\delta\br\bfe_2, & \Gammam^{\delta}
&= \Gammam+\delta\br\bfe_2, & \Gammap^{\delta} &:=
\Gammap+\delta\br\bfe_2, & \Gammao^{\delta} &=
\Gammao+\delta\br\bfe_2,
\end{align*}
where $\bfe_2$ is the unit vector in the direction of the
$x_2$--axis. Suppose that $\delta$ is so small positive number
that the curve $\Gammam^{\delta}$ still lies below the profile
$P$, which means that $P\subset\{(x_1,y_2)\in\R^2;\ y_2>x_2,\
(x_1,x_2)\in\Gammam^{\delta}\}$. (Recall that $P=\overline{{\rm
Int}\, \Gammaw}$, see Fig.~1.) Denote by $\Omega^{\delta}$ the
domain bounded by the curves $\Gammai^{\delta}$,
$\Gammam^{\delta}$, $\Gammao^{\delta}$, $\Gammap^{\delta}$ and
$\Gammaw$. Denote by $\bfv^{\delta}$ the function, defined by the
formulas
\begin{equation}
\bfv^{\delta}(x_1,x_2)\ :=\ \left\{ \begin{array}{ll}
\bfv(x_1,x_2) & \mbox{for}\ (x_1,x_2)\in \Omega^{\delta}\cap\Omega, \\
[2pt]  \bfv(x_1,x_2-\tau) & \mbox{for}\ (x_1,x_2)\in
\Omega^{\delta} \smallsetminus\Omega. \end{array} \right.
\label{3.27}
\end{equation}
Furthermore, denote
\begin{align*}
\bbF^{\delta}(x_1,x_2)\ &:=\ \left\{ \begin{array}{ll}
\bbF(x_1,x_2) & \mbox{for}\ (x_1,x_2)\in \Omega^{\delta}\cap\Omega, \\
[2pt] \bbF(x_1,x_2-\tau) & \mbox{for}\ (x_1,x_2)\in\Omega^{\delta}
\smallsetminus\Omega, \end{array} \right. \\
\bfg_*^{\delta}(x_1,x_2)\ &:=\ \left\{ \begin{array}{ll}
\bfg_*(x_1,x_2) & \mbox{for}\ (x_1,x_2)\in \Omega^{\delta}\cap\Omega, \\
[2pt] \bfg_*(x_1,x_2-\tau) & \mbox{for}\
(x_1,x_2)\in\Omega^{\delta} \smallsetminus\Omega. \end{array}
\right.
\end{align*}
Let the spaces $\Vsdelta{r}$ and $\Vsddelta{r}$ be defined by
analogy with $\Vs{r}$ $\Vsd{r}$, respectively, and let operator
$\cA_r^{\delta}$ be defined in the same way as $\cA_r$, with the
only difference that it acts from $\Vsdelta{r}$ to $\Vsddelta{r}$.
Obviously, $\bbF^{\delta}\in W^{1,r}(\Omega^{\delta})^{2\times 2}$
and $\|\bbF^{\delta}\|_{1,r;\, \Omega^{\delta}}=\|\bbF\|_{1,r}$.
Similarly, the function $\bfg_*^{\delta}$ has the same norm and
properties in $\Omega^{\delta}$ as the function $\bfg_*$ in
$\Omega$. Let the functionals $\bfF^{\delta}$ and $\bfG^{\delta}$
in the dual space $\Vsddelta{r}$ be defined by analogous formulas
as $\bfF$ and $\bfG$ in (\ref{2.3}) and (\ref{2.4}).

Our next claim is to show that $\bfv^{\delta}\in\Vsdelta{r}$ and
$\nu\cA_r^{\delta}\bfv^{\delta}=\bfF^{\delta}+\nu\,
\bfG^{\delta}$. Since $\bfv\in\Vs{r}$, there exists a sequence
$\{\bfv_n\}$ in $\cCs$, such that $\bfv_n\to\bfv$ in the norm of
$\bfW^{1,r}(\Omega)$. Let $\bfv_n^{\delta}$ be defined by
analogous formulas as $\bfv^{\delta}$. Then $\bfv_n^{\delta}\in
\boldsymbol{\cC}_{\sigma}^{\infty}(\Omega^{\delta})$ and
$\bfv_n^{\delta}\to\bfv^{\delta}$ in
$\bfW^{1,r}(\Omega^{\delta})$. This confirms that
$\bfv^{\delta}\in\Vsdelta{r}$. Furthermore, let $\bfw\in\Vs{r'}$
and $\bfw^{\delta}\in\Vsdelta{r'}$ be related in the same way as
$\bfv$ and $\bfv^{\delta}$. Then, denoting by $\langle\, .\, ,\,
.\, \rangle_{[\Vsddelta{r},\Vsddelta{r}]}$ the duality pairing
between $\Vsddelta{r}$ and $\Vsdelta{r'}$, we have
\begin{align*}
\langle\nu\cA_r^{\delta} & \bfv^{\delta}
,\bfw^{\delta}\rangle_{[\Vsddelta{r},\Vsddelta{r}]}\ =\
\nu\int_{\Omega^{\delta}}\nabla\bfv^{\delta}:\nabla\bfw^{\delta}\;
\rmd\bfx \\
&=\ \nu \int_{\Omega^{\delta}\cap\Omega}\nabla\bfv^{\delta}:\nabla
\bfw^{\delta}\; \rmd\bfx+\nu\int_{\Omega^{\delta}\smallsetminus
\Omega} \nabla\bfv^{\delta}:\nabla\bfw^{\delta}\; \rmd\bfx \\
&=\ \nu \int_{\Omega^{\delta}\cap\Omega}\nabla\bfv:\nabla\bfw\;
\rmd\bfx+\nu\int_0^{\delta} \int_{\Gammap+s\,
\rme_2}\nabla\bfv^{\delta}:\nabla\bfw^{\delta}\;
\rmd l\, \rmd s \\
&=\ \nu\int_{\Omega^{\delta}\cap\Omega}\nabla\bfv:\nabla\bfw\;
\rmd\bfx+\nu\int_0^{\delta} \int_{\Gammam+s\, \rme_2}\nabla\bfv:
\nabla\bfw\; \rmd l\, \rmd s \\ \noalign{\vskip 4pt}
&=\ \nu\int_{\Omega}\nabla\bfv:\nabla\bfw\; \rmd\bfx\ =\
\langle\nu\cA_r\bfv,\bfw \rangle_{\sigma}\ =\
\langle\bfF,\bfw\rangle_{\sigma}+\langle\bfG, \bfw\rangle_{\sigma}
\\ \noalign{\vskip 2pt}
&=\ -\int_{\Omega}\bbF:\nabla\bfw\; \rmd\bfx+
\int_{\Omega}\nabla\bfg_*:\nabla\bfw\; \rmd\bfx\ =\
-\int_{\Omega^{\delta}} \bbF^{\delta}:\nabla\bfw^{\delta}\;
\rmd\bfx+\int_{\Omega^{\delta}}\nabla\bfg_*^{\delta}:
\nabla\bfw^{\delta}\; \rmd\bfx \\ \noalign{\vskip 4pt}
&=\ \langle\bfF^{\delta},\bfw^{\delta}\rangle_{[\Vsddelta{r},
\Vsddelta{r}]}+\langle\bfG^{\delta},\bfw^{\delta}
\rangle_{[\Vsddelta{r},\Vsddelta{r}]}.
\end{align*}
As this holds for all $\bfw^{\delta}\in\Vsdelta{r'}$, we observe
that $\nu\cA_r^{\delta}\bfv^{\delta}=\bfF^{\delta}+\nu\,
\bfG^{\delta}$.

Denote $(\Omega')^{\delta/2}:=\Omega'\cup\{(x_1,y_2)\in\R^2;\
x_2\leq y_2<x_2+\frac{1}{2}\delta$ for $(x_1,x_2)\in\Gammap\}$.
Then $(\Omega')^{\delta/2}$ is a sub-domain of $\Omega^{\delta}$,
such that $\Gammap^0\subset(\Omega')^{\delta}$. The statements of
Lemma \ref{L9} now follow from Corollary \ref{C1}, applied to the
equation $\nu\cA^{\delta}\bfv^{\delta}=
\bfF^{\delta}+\nu\br\bfG^{\delta}$ in domain $\Omega^{\delta}$,
where we consider $(\Omega')^{\delta/2}$ instead of $\Omega'$.
\end{proof}

\vspace{4pt} \noindent
{\bf Completion of the proof of Theorem \ref{T2} in the case
$r\geq 2$.} \ An analogue of Lemma \ref{L9} also holds if one
considers $\Omega'$, satisfying the condition
$\Gammam\subset\partial\Omega'$ instead of
$\Gammap\subset\partial\Omega'$. This and Lemmas
\ref{L7}--\ref{L9} complete the proof of the statements 1) and 5)
of Theorem \ref{T2}.

As the function $\bfu^{\delta}:=\bfg_*^{\delta}+\bfv^{\delta}$,
where $\bfg_*^{\delta}$ and $\bfv^{\delta}$ are the functions from
the proof of Lemma \ref{L9}, belongs to
$\bfW^{2,r}((\Omega')^{\delta})$, the trace of
$\nabla\bfu^{\delta}$ on $\Gammap$ belongs to
$W^{1-1/r,r}(\Gammap)^{2\times 2}$. This and the relation between
the functions $\bfu^{\delta}$ and $\bfu$ (following from the
definition of $\bfv^{\delta}$ and $\bfg_*^{\delta}$) implies that
the trace of $\nabla\bfu$ on $\Gammap$ ``from below'' (i.e.~from
$\Omega$) equals the trace of $\nabla\bfu$ on $\Gammam$ ``from
above'' (i.e.~again from $\Omega$). This implies the validity of
the condition of periodicity (\ref{1.6}). The validity of
condition (\ref{1.7}) can be proven by means of the same
arguments.

\vspace{4pt} \noindent
{\bf The case $1<r<2$.} \ There exist sequences $\{\bff^n\}$,
$\{\bfh^n\}$ and $\bfg^n\}$ in $\bfL^2(\Omega)$,
$\bfW^{1/2,2}_{\rm per}(\Gammao)$ and $\bfW^{3/2,2}(\Gammai)$,
respectively, such that $\bff^n\to\bff$ in $\bfL^r(\Omega)$,
$\bfh^n\to\bfh$ in $\bfW^{1-1/r,r}_{\rm per}(\Gammao)$ and
$\bfg^n\to\bfg$ in $\bfW^{2-1/r,r}_{\rm per}(\Gammai)$ for
$n\to\infty$. Let $\bbF^n$ and $\bfg_*^n$ be the functions, given
by by Lemma \ref{L6}, Lemma \ref{L3} and Remark \ref{R2} in case
that we consider $\bff^n$, $\bfh^n$ and $\bfg^n$ instead of
$\bff$, $\bfh$ and $\bfg$, respectively. Let the functionals
$\bfF^n$ and $\bfG^n$ (corresponding to $\bbF^n$ and $\bfg_*^n$)
be defined by formulas (\ref{2.3}) and (\ref{2.4}), respectively.
Then it follows from \cite[Theorem 2]{TNe4}, and also from the
first part of this proof (where we assumed that $r\geq 2$), that
the unique solution $\bfv^n$ of the equation
$\nu\cA_2\bfv^n=\bfF^n+\nu\br\bfG^n$ belongs to
$\Vs{r}\cap\bfW^{2,r}(\Omega)$ and the associated pressure $p^n$
lies in $W^{1,r}(\Omega)$, the functions $\bfu^n:=\bfg_*^n+\bfv^n$
and $p^n$ satisfy equations (\ref{1.1}) (with $\bff^n=\div\bbF^n$)
and (\ref{1.2}) a.e.~in $\Omega$, $\bfu^n$, $p^n$ satisfy boundary
conditions (\ref{1.3}), (\ref{1.4}) and (\ref{1.8}) in the sense
of traces on $\Gammai$, $\Gammaw$ and $\Gammao$, respectively,
$\bfu^n$, $p^n$ satisfy the conditions of periodicity
(\ref{1.5})--(\ref{1.7}) in the sense of traces on $\Gammam$ and
$\Gammap$ and
\begin{equation}
\|\bfu^n\|_{2,2}+\|\nabla p^n\|_2\ \leq\ c\, \bigl( \|\bff^n\|_2+
\|\bfg^n\|_{3/2,2;\, \Gammai}+\|\bfh^n\|_{1/2,2;\, \Gammao}
\bigr), \label{3.28}
\end{equation}
where $c=c(\nu,\Omega)$. However, the estimate (\ref{3.11}) does
not follow from (\ref{3.28}) by the limit transition for
$n\to\infty$, because the norms $\|\bff^n\|_2$,
$\|\bfg^n\|_{3/2,2;\, \Gammai}$ and $\|\bfh^n\|_{1/2,2;\,
\Gammao}$ may tend to infinity if $n\to\infty$. Nevertheless,
repeating the procedures from the proofs of Lemmas
\ref{L7}--\ref{L9}, we also derive that
\begin{equation}
\|\bfu^n\|_{2,r}+\|\nabla p^n\|_r\ \leq\ c\, \bigl( \|\bff^n\|_r+
\|\bfg^n\|_{2-1/r,r;\, \Gammai}+\|\bfh^n\|_{1-1/r,r;\, \Gammao}
\bigr), \label{3.29}
\end{equation}
where $c=c(\nu,\Omega,r)$. The limit transition for $n\to\infty$
yields (\ref{3.11}). \hfill $\square$

\medskip \noindent
{\bf Acknowledgement.} This work was supported by European
Regional Development Fund-Project ``Center for Advanced Applied
Science'' No.~CZ.02.1.01/0.0/0.0/16\_019/0000778.

\bigskip \medskip
\hspace{1.2pt} \begin{tabular}{ll} Author's address: \quad &
Tom\'a\v{s} Neustupa \\ & Czech Technical University \\ & Faculty
of
Mechanical Engineering \\ & Department of Technical Mathematics \\
& Karlovo n\'am.~13, 121 35 Praha 2 \\ & Czech Republic \\ &
e-mail: \ tomas.neustupa@fs.cvut.cz
\end{tabular}


\begin{thebibliography}{13}
\itemsep=-2pt

\bibitem{Ag}
S.~Agmon: \textit{Lectures on Elliptic Boundary Value Problems.}
Van Nostrand Comp., New York 1965.


\bibitem{AcAmCoGh}
P. Acevedo, Ch. Amrouche, C. Conca, A. Ghosh: Stokes and
Navier-Stokes equations with Navier boundary conditions. {\it C.
R. Acad. Sci. Paris,} Ser. I 357 (2019), 

\bibitem{AmSe}
Ch.~Amrouche, N.~El Houda Seloula: On the Stokes equation with
Navier--type boundary conditions. \textit{Diff.~Equations \&
Appl.} \textbf{3} (2011), 4, 581-607.



\bibitem{BrFa}
C.~H.~Bruneau, P.~Fabrie: New efficient boundary conditions for
incompressible Navier--Stokes equations: A well--posedness result.
\textit{Math.~Modelling and Num.~Analysis} \textbf{30} (1996), 7,
815--840.

\bibitem{ChOsQi}
G.~Q.~Chen, D.~Osborne, Z.~Qian: The Navier-Stokes equations with
the kinematic and vorticity boundary conditions on non--flat
boundaries. \textit{Acta Math.~Sci.} \textbf{29}B (2009), no.~4,
919--948.

\bibitem{ChQi}
G.~Q.~Chen, Z.~Qian: A study of the Navier-Stokes equations with
the kinematic and Navier boundary conditions. \textit{Indiana
Univ.~Math.~J.} \textbf{59} (2010), no.~2, 721--760.

\bibitem{Dau}
{M.~Dauge:} Stationary Stokes and Navier--Stokes systems on two--
or three-dimensional domains with corners. Part I: linearized
equations. \textit{SIAM J.~Math.~Anal.} \textbf{20} (1989),
74--97.

\bibitem{DFF}
V.~Dolej\v{s}\'{\i}, M.~Feistauer, J.~Felcman: Numerical
simulation of compressible viscous flow through cascades of
profiles. \textit{Z.~f\"ur Angew.~Math.~Mech.} \textbf{76} (1996),
301--304.

\bibitem{EsGh}
M.~Escobedo, A.~Ghosh: Semigroup theory for the Stokes operator
with Navier boundary condition in $L^p$ spaces. ArXiv:
1808.02001v1 [math.AP] 6 Aug 2018.


\bibitem{FeNe1}
M.~Feistauer, T.~Neustupa: On some aspects of analysis of
incompressible flow through cascades of profiles. \textit{Operator
Theory, Advances and Applications,} Vol.~147, Birkhauser, Basel,
2004, 257--276.

\bibitem{FeNe2}
M.~Feistauer, T.~Neustupa: On non-stationary viscous
incompressible flow through a cascade of profiles.
\textit{Math.~Meth.~Appl.~Sci.} \textbf{29} (2006), No.~16,

\bibitem{FeNe3}
M.~Feistauer, T.~Neustupa: On the existence of a weak solution of
viscous incompressible flow past a cascade of profiles with an
arbitrarily large inflow. \textit{J.~Math.~Fluid Mech.}
\textbf{15} (2013), 701--715.

\bibitem{Ga}
G.~P.~Galdi: \textit{An Introduction to the Mathematical Theory of
the Navier--Stokes Equations, Steady State Problems.}
Springer--Verlag, New York--Berlin--Heidelberg 2011.

\bibitem{GeHeHi}
M.~Geissert, H.~Heck, M.~Hieber: On the equation $\div u = g$ and
Bogovskii's operator in Sobolev spaces of negative order.


\bibitem{Glow}
R.~Glowinski: \textit{Numerical Methods for Nonlinear
Va\-ria\-tional Prob\-lems.} Springer--Verlag, New
York--Berlin--Heidelberg--Tokyo, 1984.

\bibitem{Gr1}
P.~Grisvard: \textit{Elliptic Problems in non--Smooth Domains.}
Pitman Advanced Publishing Program, Boston--London--Melbourne
1985.

\bibitem{Gr2}
P.~Grisvard: Singularit\'es des solutions du probl\'eme de Stokes
dans un polygone. \textit{Universit\'e de Nice,} 1979.

\bibitem{HeRaTu}
J.~G.~Heywood , R.~Rannacher and S.~Turek: Artificial boundaries
and flux and pressure conditions for the incompressible
Navier-Stokes equations. \textit{Int.~J.~for Numerical Methods in
Fluids} \textbf{22} (1996), 325--352.


\bibitem{KeOs}
{R.~B.~Kellog, J.~E.~Osborn:} A regularity result for the Stokes
problem in a convex polygon. \textit{J.~Functional Analysis}
\textbf{21} (1976), 397--431.


\bibitem{KLP}
K.~Kozel, P.~Louda, J.~P\v{r}\'{\i}hoda: Numerical solution of
turbulent flow in a turbine cascade.
\textit{Proc.~Appl.~Math.~Mech.} \textbf{6} (2006), 


\bibitem{KMPT}
T.~Kato, M.~Mitrea, G.~Ponce, M.~Taylor: Extension and
representation of divergence--free vector fields on bounded
domains. \textit{Mathematical Research Letters} \textbf{7} (2000),

\bibitem{KraNe1}
S.~Kra\v{c}mar, J.~Neustupa: Modelling of flows of a viscous
incompressible fluid through a channel by means of variational
inequalities. \textit{ZAMM} \textbf{74}, No.~6, 

\bibitem{KraNe2}
S.~Kra\v{c}mar, J.~Neustupa: A weak solvability of a steady
variational inequality of the Navier-Stokes type with mixed
boundary conditions. \textit{Nonlinear Analysis} \textbf{47},
No.~6, 4169--4180 (2001).

\bibitem{KraNe3}
S.~Kra\v{c}mar, J.~Neustupa: Modeling of the unsteady flow through
a channel with an artificial outflow condition by the
Navier--Stokes variational inequality. \textit{Math.~Nachrichten}
\textbf{291} (2018), Issue 11--12, 1--14.

\bibitem{KuSka}
P.~Ku\v{c}era, Z.~Skal\'ak, Solutions to the Navier--Stokes
equations with mixed boundary conditions. \textit{Acta
Appl.~Math.}~\textbf{54}, No.~3, 275--288 (1998).

\bibitem{Ku}
P.~Ku\v{c}era: Basic properties of the non-steady Navier--Stokes
equations with mixed boundary conditions ina bounded domain.
\textit{Ann.~Univ.~Ferrara} \textbf{55}, 289--308 (2009).

\bibitem{KuBe}
P.~Ku\v{c}era, M.~Bene\v{s}: Solution to the Navier--Stokes
equatons with mixed boundary conditions in two-dimensional bounded
domains. 

\bibitem{La}
O.~A.~Ladyzhenskaya: \textit{The Mathematical Theory of Viscous
Incompresible Flow.} Gordon and Breach Science Publishers, Now
York 1969.





\bibitem{TNe1}
T.~Neustupa: Question of existence and uniqueness of solution for
Navier--Stokes Equation with linear ``do-nothing'' type boundary
condition on the outflow. \textit{Lcture Notes in Computer
Science} \textbf{5434} (2009), 431--438.

\bibitem{TNe2}
T.~Neustupa: The analysis of stationary viscous incompressible
flow through a rotating radial blade machine, existence of a weak
solution. \textit{Applied  Math.~and Computation} \textbf{219}
(2012), 3316--3322.

\bibitem{TNe3}
T.~Neustupa: A steady flow through a plane cascade of profiles
with an arbitrarily large inflow: the mathematical model,
existence of a weak solution. \textit{Applied Math.~and
Computation} \textbf{272} (2016) 687--691.

\bibitem{TNe4}
T.~Neustupa: The maximum regularity property of the steady Stokes
problem associated with a flow through a profile cascade.
Submitted, https://arxiv.org/abs /2006.15651

\bibitem{TNe5}
T.~Neustupa: The weak Stokes problem associated with a flow
through a profile cascade in $L^r$-- framework. Submitted,
https://arxiv.org/abs/2009.08234

\bibitem{Te}
R.~Temam: \textit{Navier--Stokes Equations.} North--Holland,
Amsterdam--New York--Ox\-ford 1977.

\bibitem{So}
H.~Sohr: \textit{The Navier--Stokes equations. The Eelementary
Functional Analytic Approach.} Birkh\"auser Verlag,
Basel--Boston--Berlin 2001.

\bibitem{SPKF}
P.~Straka, J.~P\v{r}\'{\i}hoda, M.~Ko\v{z}\'{\i}\v{s}ek,
J.~F\"urst: Simulation of transitional flows through a turbine
blade cascade with heat transfer for various flow conditions.
\textit{EPJ Web of Conferences} \textbf{143} (2017), 02118, DOI:
10.1051/epjconf/201714302118.


\end{thebibliography}
\end{document}